\documentclass[a4paper]{amsart} 
\usepackage{enumerate}
\usepackage{amssymb}
\usepackage{bbold}
\usepackage[all]{xy}
\SelectTips{cm}{}

\usepackage[colorlinks]{hyperref}
\calclayout 
\makeatletter
\def\serieslogo@{} 
\def\@setcopyright{} 

\makeatother

\title[Dimensions of triangulated categories]{Dimensions of triangulated categories via \\ Koszul objects}

\dedicatory{Dedicated to Karin Erdmann on the occasion of her 60th birthday.}

\thanks{Bergh and Oppermann were supported by NFR Storforsk grant no. 167130, and Iyengar was partly supported by NSF grant, DMS 0602498. Iyengar's work was done at the University of Paderborn, on a visit supported by a For\-schungs\-preis awarded by the Humboldt-Stiftung.}

\author[Bergh]{Petter Andreas Bergh}
\address{Petter Andreas Bergh\\ Institutt for matematiske fag \\ NTNU \\ N-7491 Trondheim \\ Norway}
\email{bergh@math.ntnu.no} 

\author[Iyengar]{Srikanth B. Iyengar}
\address{Srikanth B. Iyengar\\ Department of Mathematics\\
University of Nebraska\\ Lincoln NE 68588\\ U.S.A.}
\email{iyengar@math.unl.edu}

\author[Krause]{\\ Henning Krause}
\address{Henning Krause\\ Institut f\"ur Mathematik\\
Universit\"at Paderborn\\ 33095 Paderborn\\ Germany.}
\email{hkrause@math.upb.de}

\author[Oppermann]{Steffen Oppermann}
\address{Steffen Oppermann\\ Institutt for matematiske fag \\ NTNU \\ N-7491 Trondheim \\ Norway}
\email{opperman@math.ntnu.no}


\theoremstyle{plain}
\newtheorem{theorem}{Theorem}[section]

\newtheorem{proposition}[theorem]{Proposition}

\newtheorem{lemma}[theorem]{Lemma}

\newtheorem{corollary}[theorem]{Corollary}

\theoremstyle{definition}

\newtheorem{example}[theorem]{Example}

\theoremstyle{remark}

\newtheorem{remark}[theorem]{Remark}

\numberwithin{equation}{section}
\newcommand{\col}{\colon}

\newcommand{\Thick}{\operatorname{thick}\nolimits}

\renewcommand{\dim}{\operatorname{dim}\nolimits}
\newcommand{\End}{\operatorname{End}\nolimits}
\newcommand{\Ext}{\operatorname{Ext}\nolimits}

\newcommand{\Hom}{\operatorname{Hom}\nolimits}

\newcommand{\id}{\operatorname{id}\nolimits}

\newcommand{\length}{\operatorname{length}\nolimits}

\newcommand{\Ker}{\operatorname{Ker}\nolimits}

\newcommand{\ann}{\operatorname{ann}\nolimits}
\newcommand{\spec}{\operatorname{Spec}\nolimits}
\newcommand{\supp}{\operatorname{supp}\nolimits}

\renewcommand{\mod}{\operatorname{\mathsf{mod}}\nolimits}

\newcommand{\gldim}{\operatorname{gl.dim}\nolimits}
\newcommand{\repdim}{\operatorname{rep.dim}\nolimits}
\newcommand{\HH}{\operatorname{HH}\nolimits}

\newcommand{\per}{\mathrm{per}}

\newcommand{\even}{\mathrm{ev}}
\newcommand{\odd}{\mathrm{odd}}

\newcommand{\xlto}[1]{\stackrel{#1}\longrightarrow}
\newcommand{\lto}{\longrightarrow}
\newcommand{\xra}{\xrightarrow}

\newcommand{\un}{\un}

\def\Si{{\Sigma}}

\def\mcU{\mathcal{U}}
\def\mcV{\mathcal{V}}

\def\bbn{\mathbb N}

\def\bbz{\mathbb Z}

\def\sfD{\mathsf D}

\def\sfT{\mathsf T}

\def\bfD{\mathbf D}

\newcommand{\bsr}{\boldsymbol{r}}
\newcommand{\bsx}{\boldsymbol{x}}

\newcommand{\fm}{\mathfrak{m}}
\newcommand{\fp}{\mathfrak{p}}
\newcommand{\fq}{\mathfrak{q}}
\newcommand{\fr}{\mathfrak{r}}
\newcommand{\ges}{{\scriptscriptstyle\geqslant}}

\newcommand{\abel}{\mathsf{Ab}}
\newcommand{\ass}{\operatorname{Ass}\nolimits}
\newcommand{\cent}{{Z}}
\newcommand{\codim}{\operatorname{codim}\nolimits}
\newcommand{\edim}{\operatorname{edim}\nolimits}
\newcommand{\cx}{\operatorname{cx}}
\newcommand{\kos}[2]{{#1/\!\!/#2}}
\newcommand{\loewy}{\operatorname{ll}}
\newcommand{\proj}{\operatorname{Proj}}
\newcommand{\msup}[1][R]{\operatorname{Supp}^{\scriptscriptstyle +}_{#1}}
\newcommand{\Supp}[1][R]{\operatorname{Supp}_{#1}}
\newcommand{\rank}{\operatorname{rank}\nolimits}
\newcommand{\thickn}[3][\sfT]{{\operatorname{thick}_{#1}^{#2}(#3)}}
\newcommand{\level}[3][\sfT]{{\operatorname{level}_{#1}^{#2}(#3)}}
\newcommand{\std}{\mathsf{D}^{b}_{\mathsf{st}}}
\newcommand{\noeth}{\operatorname{noeth}\nolimits}
\newcommand{\flnoeth}{\operatorname{noeth}^{\mathrm{fl}}\nolimits}

\begin{document}

\begin{abstract}
Lower bounds for the dimension of a triangulated category are
provided.  These bounds are applied to stable derived categories of
Artin algebras and of commutative complete intersection local
rings. As a consequence, one obtains bounds for the representation
dimensions of certain Artin algebras.
\end{abstract} 

\maketitle

\section{Introduction}
A notion of dimension for a triangulated category was introduced by Rouquier in \cite{Rq:dim}. Roughly speaking, it corresponds to the minimum number of steps needed to generate the category from one of its objects. Consideration of this invariant has been critical to some recent developments in algebra and geometry: Using the dimension of the stable category of an exterior algebra on a $d$-dimensional vector space, Rouquier \cite{Rq:ext} proved that the representation dimension of the exterior algebra is $d+1$, thereby obtaining the first example of an algebra with representation dimension more than three.

On the other hand, Bondal and Van den Bergh~\cite{BV} proved that any cohomological finite functor on the bounded derived category of coherent sheaves on a smooth algebraic variety over a field is representable, by establishing that that triangulated category has finite dimension.

In this paper we establish lower bounds for the dimension of a triangulated category and discuss some applications. We make systematic use of the graded-commutative structure of the triangulated category -- in particular, Koszul objects -- arising from its graded center; see Section~\ref{sec:koszulobjects}. 

A triangulated category $\sfT$ is by definition an additive $\bbz$-category equipped with a class of exact triangles satisfying various axioms \cite{V}. Here, \emph{$\bbz$-category} simply means that there is a fixed equivalence $\Si\colon\sfT\to\sfT$.

Given any additive $\bbz$-category $\sfT=(\sfT,\Si)$, we introduce a natural finiteness condition for objects of $\sfT$ as follows: Let $R=\bigoplus_{i\ges 0}R^{i}$ be a graded-commutative ring that acts on $\sfT$ via a homomorphism of graded rings $R\to \cent^*(\sfT)$ to the graded center of $\sfT$. Thus for each pair of
objects $X,Y$ in $\sfT$, the graded abelian group
\[
\Hom_\sfT^*(X,Y)=\bigoplus_{i\in\bbz}\Hom_\sfT(X,\Si^iY)
\]
is a graded $R$-module.

Now fix an object $X$ in $\sfT$ and suppose that for each $Y\in\sfT$ there exists an integer $n$ such that the following properties hold:
\begin{enumerate}[{\quad\rm(1)}]
\item the graded $R$-module $\bigoplus_{i\ges n}\Hom_\sfT(X,\Si^{i}Y)$ is noetherian, and
\item the $R^{0}$-module $\Hom_{\sfT}(X,\Si^{i}Y)$ is of finite length for $i\ge n$.
\end{enumerate}

In this case, $\Hom_\sfT^{*}(X,Y)$ has finite (Krull) dimension over $R^{\even}$, the subring of $R$ consisting of elements of even degree, which is a commutative ring. If $X$ has this finiteness property also with respect to another ring $S$, then the dimension of $\Hom^*_{\sfT}(X,Y)$ over $S$ coincides with that over $R$; see Lemma~\ref{lem:independence}. For this reason, we denote this number by $\dim\Hom^*_{\sfT}(X,Y)$. 

The main result in this work is as follows.

\begin{theorem}
\label{ithm:intro}
Let $\sfT$ be a triangulated category and $X$ an object with properties as above. One then has an inequality
\[
\dim \sfT \geq \dim\End^{*}_{\sfT}(X)-1\,.
\]
\end{theorem}

An intriguing feature of this result is that the invariant appearing on the right hand side of the inequality involves only the additive $\bbz$-structure of $\sfT$. Theorem~\ref{ithm:intro} is contained in Theorem~\ref{thm:main}.  The proof is based on a  systematic use of Koszul objects and elementary observations concerning `eventually noetherian modules'; this is inspired by the approach in \cite{AI}. Another important ingredient is a version of the `Ghost Lemma' from \cite{Pb}; see Lemma~\ref{lem:ghosts}. 

Our principal motivation for considering dimensions of triangulated categories is that it provides a way to obtain lower bounds on the representation dimension of an Artin algebra. Indeed, for a non-semisimple Artin algebra $A$ one has an inequality
\[
\repdim A \geq \dim \std(A)+2\,,
\]
where $\std(A)$ is the stable derived category of $A$, in the sense of Buchweitz~\cite{Bu}. As one application of the preceding result, we bound the representation dimension of $A$ by the Krull dimension of Hochschild cohomology.

\begin{corollary}
\label{icor:Hoch}
Let $k$ be an algebraically closed field and $A$ a finite dimensional $k$-algebra with radical $\fr$, where $A$ is not semi-simple. If $\Ext^{*}_A(A/\fr,A/\fr)$ is noetherian as a module over the Hochschild cohomology algebra $\HH^{*}(A)$ of $A$ over $k$, then
\[
\repdim A\geq \dim \HH^{*}(A)+1\,.
\]
\end{corollary}

This result is a special case of Corollary~\ref{cor:Hoch}. In Section~\ref{Applications} we present further applications of Theorem~\ref{thm:main}.

\section{Eventually noetherian modules}

Many of the arguments in this article are based on properties of `eventually noetherian modules' over graded commutative rings, introduced by Avramov and Iyengar~\cite[\S2]{AI}. In this section we collect the required results. For the benefit of the reader we provide (sketches of) proofs, although the results are well-known, and the arguments based on standard techniques in commutative algebra. For unexplained terminology the reader is referred to Bruns and Herzog~\cite{BH}.

\subsection*{Graded-commutative rings}
Let $R=\bigoplus_{i\ges 0}R^i$ be a \emph{graded-commutative ring}; thus $R$ is an $\bbn$-graded ring with the property that $rs =(-1)^{|r||s|}sr$ for any $r,s$ in $R$. Elements in a graded object are assumed to be homogeneous.

Let $M=\bigoplus_{i\in\bbz}M^i$ be a graded $R$-module. For any integer $n$, we set
\[
 M^{\ges n}=\bigoplus_{i\ges n}M^i\qquad\text{and}\qquad R^{+}= R^{\ges 1}\,.
\]
Note that $M^{\ges n}$ is an $R$-submodule of $M$, and that $R^{+}$ is an ideal in $R$.

As in \cite[\S2]{AI}, we say that $M$ is \emph{eventually noetherian} if the $R$-module $M^{\ges n}$ is noetherian for some integer $n$; we write $\noeth(R)$ for the full subcategory of the category of all graded $R$-modules, with objects the eventually noetherian modules. In this work, the focus is on eventually noetherian modules $M$ that have the additional property that $\length_{R^{0}}(M^{n})$ is finite for $n\gg 0$. The corresponding full subcategory of $\noeth(R)$ is denoted $\flnoeth(R)$. It is easy to verify that both $\noeth(R)$ and $\flnoeth(R)$ are abelian subcategories.

Recall that the \emph{annihilator} of $M$, which we denote $\ann_{R}M$, is the homogenous ideal of $R$ consisting of elements $r$ such that $r\cdot M=0$. The following remark is easily justified. It allows one, when considering eventually noetherian modules, to pass to a situation where the ring itself is noetherian.

\begin{remark}
\label{rem:passtonoeth}
Suppose that the $R$-module $M^{\ges n}$ is noetherian. Set $I=\ann_R(M^{\ges n})$.  The ring $R/I$ is then noetherian, and $M^{\ges n}$ is a finitely generated and faithful $R/I$-module.  If in addition $\length_{R^0}(M^i)$ is finite for $i\geq n$, then $(R/I)^0$ is artinian.
\end{remark}

One way to study modules over graded-commutative rings is to pass to the subring $R^{\even}$ consisting of elements of even degree, which is then a \emph{commutative graded} ring: $rs=sr$ for any $r,s$ in $R^{\even}$. In this work, this passage is facilitated by the following observation; confer the proof of \cite[Theorem 1.5.5]{BH}.

\begin{lemma}
\label{lem:reven}
Let $R$ be a graded-commutative ring, and let $M$ be an $R$-module.
\begin{enumerate}[{\quad\rm(1)}]
\item If $M$ is in $\noeth(R)$, then  it is also in $\noeth(R^{\even})$.
\item If $M$ is in $\flnoeth(R)$, then  it is also in $\flnoeth(R^{\even})$.
\end{enumerate}
\end{lemma}

\begin{proof}
Suppose $M$ is in $\noeth(R)$. By Remark~\ref{rem:passtonoeth}, one can assume $R$ is itself noetherian and $M$ a finitely generated $R$-module. It then suffices to prove that the subring $R^{\even}$ is noetherian and that $R$ is finitely generated as a module over it. Observe that there is a decomposition $R=R^{\even}\oplus R^{\odd}$  as $R^{\even}$-modules. In particular, for any ideal $I\subseteq R^{\even}$ one has $IR\cap R^{\even}=I$, and hence  $R$ noetherian implies $R^{\even}$ noetherian. By the same token, one obtains  that $R^{\odd}$, and hence also $R$, is a noetherian $R^{\even}$-module.
\end{proof}

\subsection*{Dimension over a commutative graded ring}
Let $R$ be a commutative graded ring. We recall some facts concerning the support of an $R$-module $M$, which we denote $\Supp M$. It is convenient to employ also the following notation:
\begin{align*}
  \proj R &= \{\text{$\fp$ is a homogeneous prime in $R$ with $\fp\not\supseteq R^{+}$}\} \\
  \msup M &= \{\fp\in\proj R\mid M_{\fp}\ne 0\}\\
  \ass_{R}^{+}M &=\{\fp\in\proj R\mid \Hom_{R_{\fp}}(R_{\fp}/\fp R_{\fp},M_{\fp})\ne 0\}.
\end{align*}
Evidently, $\msup M = \Supp M\cap \proj R$ and $\ass_{R}^{+}M=\ass_{R}M\cap \proj R$, where $\ass_{R}M$ is the set of associated primes of $M$. The next result is readily verified. 

\begin{lemma}
\label{lem:plus}
\pushQED{\qed}%
Let $R$ be a commutative graded ring and $M$ a graded $R$-module. 
\begin{enumerate}[{\quad\rm(1)}]
\item
For any integer $n$, one has equalities
\[
\msup {(M^{\ges n})} = \msup M\quad\text{and}\quad \ass_{R}^{+}(M^{\ges n}) = \ass_{R}^{+}M\,.
\]
\item If $L\subseteq M$ is a submodule, then 
\[
\msup L\subseteq \msup M\quad\text{and}\quad \ass_{R}^{+}L \subseteq \ass_{R}^{+}M\,.\qedhere
\]
\end{enumerate}
\end{lemma}

A module $M$ is said to be \emph{eventually zero} if $M^{\ges n}=0$ for some integer $n$. The next result is part of \cite[\S2.2]{AI}, where it is stated without proof. We give details, for the convenience of readers.

\begin{proposition}
\label{prop:evn}
Let $R$ be a commutative graded ring and $M$ an eventually noetherian $R$-module. The set $\ass_{R}^{+}M$ is finite and the conditions below are equivalent:
\begin{enumerate}[{\quad\rm(i)}]
\item
$\ass_{R}^{+} M=\emptyset$;
\item 
$\msup M=\emptyset$;
\item
$M$ is eventually zero.
\end{enumerate}
\end{proposition}

\begin{proof}
In view of Remark~\ref{rem:passtonoeth} and Lemma~\ref{lem:plus}(1), one may assume $R$ is noetherian and that $M$ is a faithful $R$-module. In this case $\ass_{R}M$ is a finite set and therefore
$\ass_{R}^{+}M$ is finite; see \cite[Theorem~6.5]{Mat}.

An ideal $\fp\in\spec R$ belongs to $\Supp M$ if and only if there exists $\fq\in\ass_RM$ with $\fq\subseteq\fp$; see \cite[Theorem~6.5]{Mat}.  From this the implications (iii) $\implies$ (ii) and (ii) $\iff$ (i) are obvious consequences.

It remains to show (ii) $\implies$ (iii). Since the $R$-module $M$ is finitely generated and faithful, one has that $\Supp M = \spec R$. Thus, $\msup M=\emptyset$ implies $R^{+}\subseteq \fp$ for each $\fp\in \spec R$, hence the ideal $R^{+}$ is nilpotent. Since $R$ is noetherian, this implies that $R$ is eventually zero, and hence also that $M$ is eventually zero.
\end{proof}

We say that an element $r\in R^{+}$ is \emph{filter-regular} on $M$ if $\Ker(M\xra{r} M)$ is eventually zero. This notion is a minor variation on a well-worn theme in commutative algebra; confer, for instance, Schenzel, Trung, and Cuong~\cite[\S2.1]{CST}.

\begin{lemma}
\label{lem:filter}
Let $R$ be a commutative graded ring and $M$ an eventually noetherian $R$-module. There then exists an element in $R^{+}$ that is filter-regular on $M$.
\end{lemma}

\begin{proof}
Proposition~\ref{prop:evn} yields that the set $\ass_{R}^{+}M$ is finite, so by prime avoidance \cite[Lemma~1.5.10]{BH} there exists an element $r$ in $R^{+}$ not contained in any prime $\fp$ in $\ass_{R}^{+}M$. This element is filter-regular on $M$.

Indeed, for $K=\Ker(M\xra{r}M)$, one has $\ass_{R}^{+}K\subseteq\ass_{R}^{+}M$; see Lemma~\ref{lem:plus}(2). However, for any $\fp$ in $\ass_{R}^{+}M$ one has $K_{\fp}=0$, since $r\not\in\fp$, and hence $\ass_{R}^{+}K=\emptyset$. Since $K$ is eventually noetherian, being a submodule of $M$, Proposition~\ref{prop:evn} applies and yields that $K$ is eventually zero.
\end{proof}

As usual, the (Krull) \emph{dimension} of a module $M$ over $R$ is the number
\[
\dim_{R}M =\sup\left\{ d\in\bbn \left| 
\begin{gathered} 
\text{there exists a chain of prime ideals} \\
\text{$\fp_0\subset \fp_1\subset \cdots \subset \fp_d$ in $\Supp M$}
\end{gathered}
\right.\right\}\,.
\]
When $M$ is in $\flnoeth(R)$ one can compute its dimension in terms of the rate of growth of its components. To make this precise, it is convenient to introduce the \emph{complexity} of a sequence of non-negative integers $(a_n)$ as the number
\[
\cx (a_n)= \inf \left\{d\in\bbn \left|\,
\begin{gathered}\text{there exists a real number $c$ such that}\\
\text{$a_n\leq cn^{d-1}$ for $n\gg0$}
\end{gathered}\right.\right\}\,.
\]
For basic properties of this notion see, for example, \cite[\S2 and Appendix]{Av}.  As usual, the set of prime ideals of $R$ containing a given ideal $I$ is denoted $\mcV(I)$.

\begin{proposition}
\label{prop:flevn}
Let $R$ be a commutative graded ring and $M\in \flnoeth(R)$.
\begin{enumerate}[{\quad\rm(1)}]
\item
If $r_{1},\dots,r_{n}$ are elements in $R^{+}$ with $n < \dim_{R}M$, then $\mcV(\bsr) \cap \msup M \neq \emptyset$.
\item
One has an equality $\dim_{R}M=\cx(\length_{R^{0}}(M^{n}))$.
\end{enumerate}
\end{proposition}

\begin{proof}
By Remark~\ref{rem:passtonoeth}, one may assume  $R$ is noetherian, $R^{0}$ is artinian, and $M$ is a faithful, finitely generated $R$-module. Part (1) then follows from the Krull height theorem; see \cite[Theorem~13.5]{Mat}, while (2) is contained in \cite[Theorem~13.2]{Mat}.
\end{proof}

\subsection*{Dimension over a graded-commutative ring}
Let $R$ be a graded-commutative ring. For each $R$-module $M$ in $\flnoeth(R)$, we introduce its \emph{dimension} as the number
\[
\dim_{R}M = \cx(\length_{R^{0}}(M^{n}))\,.
\]

It follows from Lemma~\ref{lem:reven} and Proposition~\ref{prop:flevn}(2) that this number is finite and coincides with the dimension of $M$ as a module over $R^{\even}$. This remark will be used without further comment.

\begin{proposition}
\label{prop:dimbasechange}
Let $R\to S$ be a homomorphism of graded-commutative rings and $M$ an $S$-module.  If $M$, viewed as an $R$-module by restriction of scalars, is in $\flnoeth(R)$, then $\dim_{R}M = \dim_{S}M$.
\end{proposition}

\begin{proof}
The module $M$ is in $\flnoeth(S)$ as well and therefore, by Remark~\ref{rem:passtonoeth}, one can pass to a situation where $S$ is noetherian and $M$ is a faithful $S$-module that is also noetherian over $R$. Passing to $R/I$, where $I$ is the kernel of the homomorphism $R\to S$, one may also assume that the homomorphism is
injective. Since one has injective homomorphisms of $R$-modules
\[
R\hookrightarrow S\hookrightarrow \Hom_{R}(M,M)\,,
\]
one thus obtains that the ring $R$ itself is noetherian with $R^{0}$ artinian, and that $S$ is a finitely generated $R$-module. This implies that the $R^{0}$-module $S^{0}$ is finitely-generated, and hence, for any $S^{0}$-module $N$, one has inequalities
\[
\length_{S^{0}}N \leq \length_{R^{0}}N \leq (\length_{R^{0}}S^{0})(\length_{S^{0}}N)\,.
\]
This yields $\dim_{R}M = \dim_{S}M$, as claimed.
\end{proof}

\section{Koszul objects}
\label{sec:koszulobjects}
Let $\sfT$ be a triangulated category. For any objects $X$ and $Y$ in $\sfT$, we set
\[
\Hom^{*}_{\sfT}(X,Y) = \bigoplus_{n\in\bbz}\Hom_{\sfT}(X,\Si^{n}Y) \quad\text{and}\quad
\End^{*}_{\sfT}(X) = \Hom^{*}_{\sfT}(X,X)\,.
\]
The \emph{graded center} of $\sfT$, which we denote $\cent^{*}(\sfT)$, consists in degree $n$ of natural transformations $\eta\col\id_{\sfT}\to \Si^{n}$ satisfying $\eta\Sigma = (-1)^{n}\Sigma\eta$. Composition gives $\cent^{*}(\sfT)$ a structure of a graded-commutative ring; see, for instance, \cite[\S3]{BF}, especially Lemma~3.2.1, which explains the signed commutation rule, and also \cite{Li}.

In what follows, we assume that a graded-commutative ring $R$ acts \emph{centrally on $\sfT$}, via a homomorphism $R\to \cent^{*}(\sfT)$. What this amounts to is specifying for each $X$ in $\sfT$ a homomorphism of rings $\phi_{X}\col R\to \End^{*}_{\sfT}(X)$ such that the induced $R$-module structures on $\Hom^{*}_{\sfT}(X,Y)$ coincide up to the usual sign rule:
\[
\eta\circ \phi_{X}(r) = (-1)^{|r||\eta|}\phi_{Y}(r)\circ \eta
\]
for any $\eta\in \Hom^{*}_{\sfT}(X,Y)$ and $r\in R$. 

We now recall an elementary, and extremely useful, construction.

\subsection*{Koszul objects}
Let $r$ be a homogeneous element in $R$ of degree $d=|r|$. Given an object $X$ in $\sfT$, we denote $\kos Xr$ any object that appears in an exact triangle
\begin{equation}
\label{eq:koszul}
X\xlto{r} \Si^{d}X\lto \kos Xr \lto \Si  X\,.
\end{equation}
It is well-defined up to isomorphism; we call it a \emph{Koszul object of $r$ on $X$}. 

Let $Y$ be an object in $\sfT$ and set $M=\Hom^*_{\sfT}(X,Y)$. Applying $\Hom^*_{\sfT}(-,Y)$ to the triangle above yields an exact sequence of $R$-modules:
\begin{align*}
M[d+1] \xlto{\mp r} M[1]\lto \Hom^*_{\sfT}(\kos Xr,Y)\lto M[d]\xlto{\pm r}M[0]\,.
\end{align*}
This gives rise to an exact sequence of graded $R$-modules
\begin{equation}
\label{eq:koszul-les}
0\lto (M/rM)[1] \lto \Hom^*_{\sfT}(\kos Xr, Y)\lto (0:r)_M[d]\lto 0\,,
\end{equation}
where $(0:r)_{M}$ denotes $\{m\in M\mid r\cdot m=0\}$.

Applying the functor $\Hom^*_{\sfT}(Y,-)$ results in a similar exact sequence. 

Given a sequence of elements $\bsr=r_1,\ldots,r_n$ in $R$, consider objects $X_i$ defined by
\begin{equation}
\label{eq:koszul-defn}  
X_i = \begin{cases}
X & \text{for $i=0$,}\\
\kos{X_{i-1}}{r_i} & \text{for $i\geq 1$.}
\end{cases}
\end{equation}
Set $\kos X{\bsr} = X_n$; this is a \emph{Koszul object of $\bsr$ on $X$}.  The result below is a straightforward consequence of \eqref{eq:koszul-les} and an induction on $n$; see \cite[Lemma 5.11(1)]{BIK:2008}.

\begin{lemma}
\label{lem:koszul}
\pushQED{\qed}%
Let $n\ge 1$ be an integer and set $s=2^{n}$. For any sequence of elements $\bsr= r_{1},\dots,r_{n}$ in $R^{+}$, and any object $X\in \sfT$ one has that
\[
r_{i}^{s} \cdot \Hom^{*}_{\sfT}(\kos X{\bsr},-)=0 
=  r_{i}^{s}\cdot \Hom^{*}_{\sfT}(-,\kos X{\bsr}) \quad \text{for $i=1,\dots,n$}. \qedhere
\]
\end{lemma}

The next construction quantifies the process of `building' objects out of a given object in the triangulated category $\sfT$.

\subsection*{Thickenings}
Given an object $G$ of $\sfT$ we write $\Thick_{\sfT}(G)$ for the thick subcategory of $\sfT$ generated by $G$. This subcategory has a filtration 
\[
\{0\}= \thickn 0G\subseteq \thickn 1G\subseteq \cdots \subseteq \bigcup_{n\ges 0}\thickn nG=\Thick_{\sfT}(G)
\]
where $\thickn 1G$ consists of retracts of finite direct sums of suspensions of $G$, and $\thickn nG$ consists of retracts of $n$-fold extensions of $\thickn 1G$. In the literature, the subcategory $\thickn nG$ has sometimes been denoted $\langle G\rangle_{n}$.

The next result is contained in \cite[Lemma~2.1]{Pb}. Similar results have appeared in Kelly~\cite{Ke}, Carlsson~\cite[Proof of Theorem~16]{Ca}, Christensen~\cite[Theorem~3.5]{Ch}, Beligiannis~\cite[Corollary 5.5]{Be}, Rouquier~\cite[Lemma~4.11]{Rq:dim}, and Avramov, Buchweitz, and Iyengar~\cite[Proposition~2.9]{ABI}.

\begin{lemma}[Ghost Lemma]
\label{lem:ghosts}
Let $\sfT$ be a triangulated category, and let $F,G$ be objects in $\sfT$. Suppose there exist morphisms
\[
K_{c}\xra{\theta_{c}} K_{c-1}\xra{\theta_{c-1}}\cdots \xra{\theta_{1}} K_{0} 
\]
in $\sfT$ such that the following conditions hold:
\begin{enumerate}[\quad\rm(1)]
\item $\Hom_{\sfT}^{n}(G,\theta_{i})=0$ for $n\gg 0$ and for each $i=1,\dots,c$;
\item $\Hom_{\sfT}^{n}(F,\theta_{1}\cdots \theta_{c})\ne 0$ for infinitely many $n\ge 0$.
\end{enumerate}
One then has that $F\not\in \thickn cG$. \qed
\end{lemma}

There is also a contravariant version of the Ghost Lemma, involving $\Hom_{\sfT}^{*}(-,G)$.

\begin{theorem}
\label{thm:koszul}
Let $\sfT$ be a triangulated category and $R$ a graded-commutative ring acting centrally on it.  Let $X,Y$ be objects in $\sfT$ with the property that the $R$-module $\Hom^{*}_{\sfT}(X,Y)$ is in $\flnoeth(R)$.

For any $c<\dim_{R}\Hom^{*}_{\sfT}(X,Y)$ there exist elements $r_{1},\dots,r_{c}$ in $(R^{\even})^{+}$ with
\[
\kos X{\bsr} \not\in \thickn cX\quad\text{and}\quad
\kos Y{\bsr} \not\in \thickn cY\,.
\]
\end{theorem}

\begin{remark}
In the language of levels, introduced in \cite[\S2.3]{ABIM}, the conclusion of the preceding theorem reads:
\[
\level X{\kos X{\bsr}}> c \quad\text{and}\quad \level Y{\kos Y{\bsr}}> c\,.
\]
This formulation is sometimes more convenient to use in arguments.
\end{remark}

\begin{proof} 
The plan is to apply the Ghost Lemma. 

By Lemma~\ref{lem:reven} one can assume that $R=R^{\even}$, and in particular that the graded ring $R$ is commutative. The $R$-module $\Hom^{*}_{\sfT}(X,Y)$ is in $\flnoeth(R)$ and hence so are $\Hom^{*}_{\sfT}(X,\kos Y{\bsx})$ and $\Hom^{*}_{\sfT}(\kos X{\bsx},Y)$, for any finite sequence $\bsx$ of elements in $R$; this
can be checked using \eqref{eq:koszul-les} and an induction on the length of $\bsx$. 

Set $s=2^{c}$. Using the observation in the previous paragraph and Lemma~\ref{lem:filter}, one can find, by iteration, elements $r_{1},\dots,r_{c}$ in $R^{+}$ such that for $i=1,\dots,c$ the element $r_{i}$ is filter-regular on the $R$-module
\[
\Hom^{*}_{\sfT}(X,\kos Y{\{r_{1}^{s},\dots,r_{i-1}^{s}\}}) \oplus 
\Hom^{*}_{\sfT}(\kos X{\{r_{1}^{s},\dots,r_{i-1}^{s}\}},Y)\,
\]
Equivalently, the element $r_{i}$ is filter-regular on each of the direct summands above.

We now verify that $\kos X{\bsr}$ is not in $\thickn cX$. 

Set $K_{0}=Y$, set $K_{i} = \Si^{-i}\big(\kos Y{\{r_{1}^{s},\dots,r_{i}^{s}\}}\big)$ for $i=1,\dots,c$, and let
\begin{equation}
\label{eq:proof}
K_{i}\xlto{\theta_{i}}K_{i-1}\xra{\pm r_{i}^{s}} \Si^{s|r_{i}|}K_{i-1}\to \Si K_{i}\,,
\end{equation}
be the exact triangle obtained (by suitable suspension) from the one in \eqref{eq:koszul}.

We claim that for each $i=1,\dots,c$ the following properties hold:
\begin{enumerate}[\quad\rm(1)]
\item $\Hom_{\sfT}^{n}(X,\theta_{i})=0$ for $n\gg 0$;
\item $\Hom_{\sfT}^{*}(\kos X{\bsr},\theta_{i})$ is surjective;
\item $\Hom_{\sfT}^{n}(\kos X{\bsr},K_{0})\ne 0$ for infinitely many $n\ge 0$.
\end{enumerate}

Indeed, for each $W\in \sfT$ the triangle \eqref{eq:proof} induces  an exact sequence 
\[
\Hom^{*}_{\sfT}(W,K_{i}) \xra{\Hom^{*}_{\sfT}(W,\theta_{i})} \Hom^{*}_{\sfT}(W,K_{i-1}) 
   \xra{\pm r_{i}^{s}} \Hom^{*}_{\sfT}(W,K_{i-1})[s|r_{i}|]
\]
of graded $R$-modules.

(1) With $W=X$ in the sequence above, $r_{i}$ is filter-regular on $\Hom^{*}_{\sfT}(X,K_{i-1})$, by choice, and hence so is $r_{i}^{s}$. This proves the claim. 

(2) Set $W=\kos X{\bsr}$ in the exact sequence above, and note that $r_{i}^{s}$ annihilates  $\Hom^{*}_{\sfT}(\kos X{\bsr},K_{i-1})$, by Lemma~\ref{lem:koszul}.  

(3) Recall that $K_{0}=Y$. It suffices to prove that one has an equality
\[
\msup \Hom_{\sfT}^{*}(\kos X{\bsr},Y)= \mcV(\bsr)\cap\msup \Hom_{\sfT}^{*}(X,Y)\,.
\]
For then the choice of $c$ ensures that the set above is non-empty, by Proposition~\ref{prop:flevn}(1), and hence $\Hom_{\sfT}^{*}(\kos X{\bsr},Y)$ is not eventually zero, by Proposition~\ref{prop:evn}.

The equality above can be established as in the proof of \cite[Proposition 3.10]{AI}: By induction on the length of the sequence $\bsr$, it suffices to consider the case where $\bsr=r$. Setting $M=\Hom^{*}_{\sfT}(X,Y)$, it follows from \eqref{eq:koszul-les} that one has an equality
\[
\msup \Hom_{\sfT}^{*}(\kos X{r},Y) = \msup (M/rM) \cup \msup (0:r)_{M}\,.
\]
It then remains to note that one has
\[
\msup (M/rM) = \msup M \cap \mcV(r)\quad\text{and}\quad \msup (0:r)_{M}\subseteq \msup M\cap \mcV(r)\,,
\]
where the equality holds because one has $M/rM= M\otimes_{R}R/Rr$, while the inclusion holds because $(0:r)_{M}$ is a submodule of $M$ annihilated by $r$.

This justifies claims (1)--(3) above. 

Observe that (2) and (3) imply that $\Hom_{\sfT}^{*}(\kos X{\bsr},\theta_{1}\cdots\theta_{c})$ is not eventually zero. Therefore, the Ghost Lemma yields $\kos X{\bsr}\not\in \thickn cX$, as desired.

A similar argument, employing the contravariant version of the Ghost Lemma, establishes that $\kos Y{\bsr}$ is not in $\thickn cY$.
\end{proof}

\section{The dimension of a triangulated category}
\label{Dimension of a triangulated category}

The \emph{dimension} of a triangulated category $\sfT$ is the number
\[
\dim \sfT = \inf \{n\in\bbn \mid \text{there exists a $G\in\sfT$ with $\thickn {n+1}G=\sfT$}\}.
\]
Evidently, if $\dim \sfT$ is finite there exists an object $G$ with $\Thick_{\sfT}(G)=\sfT$; we call such an object $G$ a \emph{generator} for $\sfT$. The dimension of $\sfT$ can be infinite even if it has a
generator.

\begin{lemma}
\label{lem:dimsym}
Let $\sfT$ be a triangulated category and $R$ a graded-commutative ring acting centrally on it. If
$G$ is a generator for $\sfT$, then for each object $X$ in $\sfT$ one has equalities
\[
\msup[{R^{\even}}]\Hom^{*}_{\sfT}(X,G) = \msup[{R^{\even}}] \End^{*}_{\sfT}(X)
 = \msup[{R^{\even}}] \Hom^{*}_{\sfT}(G,X)\,.
\]
\end{lemma}

\begin{proof}
We may assume $R=R^{\even}$. Using the fact that localization is an exact functor, it is easy verify that for any subset $\mcU$ of $\spec R$ the subcategory 
\[
\{Y\in\sfT\mid \msup\Hom^{*}_{\sfT}(X,Y)\subseteq \mcU\}
\]
of $\sfT$ is thick. Since $X$ is in $\Thick_{\sfT}(G)$, one thus obtains an inclusion
\[
\msup\End^{*}_{\sfT}(X) \subseteq \msup \Hom^{*}_{\sfT}(X,G)\,.
\]
The reverse inclusion holds because $R$ acts on $\Hom^{*}_{\sfT}(X,G)$ via a homomorphism of rings $R\to\End^{*}_{\sfT}(X)$. This settles the first equality.

A similar argument gives the second one.
\end{proof}

\begin{theorem}
\label{thm:main}

Let $\sfT$ be a triangulated category and $R$ a graded-commutative ring acting centrally on it. If an object $X\in\sfT$ is such that the $R$-module $\Hom^{*}_{\sfT}(X,G)$, or $\Hom^{*}_{\sfT}(G,X)$, is in
$\flnoeth(R)$, for some generator $G$, then one has an inequality
\[
\dim \sfT\geq \dim_{R} \End^{*}_{\sfT}(X) - 1\,.
\]
\end{theorem}

\begin{proof} Suppose that the $R$-module $\Hom^{*}_{\sfT}(X,G)$ is in $\flnoeth(R)$. The full subcategory of $\sfT$ with objects
\[
\{Y\in \sfT\mid \Hom^{*}_{\sfT}(X,Y)\in \flnoeth(R)\}
\]
is thick. Since it contains $G$ it coincides with $\sfT$, so one may assume that $G$ is an arbitrary generator for $\sfT$. For $c=\dim_{R}\Hom^{*}_{\sfT}(X,G)-1$, Theorem~\ref{thm:koszul} yields a Koszul object, $\kos G{\bsr}$, not contained in $\thickn{c}G$. This implies the inequality below:
\[
\dim \sfT\geq \dim_{R} \Hom^{*}_{\sfT}(X,G) -1 = \dim_{R} \End^{*}_{\sfT}(X) -1 \,;
\]
the  equality is by Lemma~\ref{lem:dimsym}. The other case is handled in the same way.
\end{proof}

In the theorem, the number $\dim_{R}\End^{*}_{\sfT}(X)$ is independent of the ring $R$, in a sense explained in the following lemma. These results together justify Theorem~\ref{ithm:intro}.

\begin{lemma}
\label{lem:independence}
Let $\sfT$ be an additive $\bbz$-category and let $X,Y$ be objects in $\sfT$. Suppose that there are graded-commutative rings $R$ and $S$ acting centrally on $\sfT$ such that $\Hom^{*}_{\sfT}(X,Y)$ is in both $\flnoeth(R)$ and $\flnoeth(S)$. One then has an equality
\[
\dim_{R}\Hom^{*}_{\sfT}(X,Y)=\dim_{S}\Hom^{*}_{\sfT}(X,Y)\,.
\]
\end{lemma}

\begin{proof}
\pushQED{\qed} %
Indeed, the graded tensor product $R\otimes_{\bbz}S$ is a graded-commutative ring, and one has natural homomorphisms of graded rings $R\to R\otimes_{\bbz}S\gets S$. The central actions of $R$ and $S$ on $\sfT$ extend to one of the ring $R\otimes_{\bbz}S$, and $\Hom^{*}_{\sfT}(X,Y)$ is in $\flnoeth(R\otimes_{\bbz}S)$.
Proposition~\ref{prop:dimbasechange}, applied to the preceding homomorphisms, now yields equalities
\[
\dim_{R}\Hom^{*}_{\sfT}(X,Y)=\dim_{R\otimes_{\bbz}S}\Hom^{*}_{\sfT}(X,Y)=\dim_{S}\Hom^{*}_{\sfT}(X,Y)\,.
\qedhere
\]
\end{proof}

\begin{remark}
\label{rem:flnoeth}
The preceding result suggests that one should consider the full subcategory of objects $X$ in $\sfT$ with the property that, for some ring $R$ acting centrally on $\sfT$ and all $Y\in \sfT$, one has $\Hom^{*}_{\sfT}(X,Y)\in \flnoeth(R)$; let us denote it $\flnoeth(\sfT)$. Arguing as in the proof of Lemma~\ref{lem:independence}, it is not difficult to prove that $\flnoeth(\sfT)$ is precisely the subcategory $\flnoeth(\cent^{\ges 0}(\sfT))$, where $\cent^{\ges0}(\sfT)$ is the non-negative part of the graded center of $\sfT$. This implies, for instance, that $\flnoeth(\sfT)$ is a thick subcategory of $\sfT$, and also that one has an `intrinsic' notion of dimension for objects in this subcategory. Thus, one could state the main results of this section
without involving an `external' ring $R$. In practice, however there are usually more convenient choices than $\cent^{\ges 0}(\sfT)$, for a ring $R$ acting centrally on $\sfT$.
\end{remark}

\subsection*{Cohomological functors}
There are also versions of Theorem~\ref{thm:main} which apply to cohomological functors. In order to explain this, let $\sfT$ be a triangulated category and $H\col\sfT\to \abel$ a cohomological functor to the category of abelian groups.  Let $R$ be a graded-commutative ring that acts centrally on $\sfT$. The graded abelian group
\[
H^{*}(Y)=\bigoplus_{n\in\bbz} H(\Si^{n}Y)
\]
then has a natural structure of a graded $R$-module.  

Assume that there exists a generator $G$ of $\sfT$ such that the $R$-module $H^{*}(G)$ is noetherian and the $R^{0}$-module $H^{i}(G)$ has finite length for each $i$. One can check, as in Lemma~\ref{lem:independence}, that in this case, for any $Y\in \sfT$, the dimension of the $R$-module $H^{*}(Y)$ is finite and independent
of $R$; denote it $\dim H^{*}(Y)$.

\begin{theorem}
\label{thm:rep}
Let $\sfT$ be a triangulated category, and assume that idempotents in $\sfT$ split.  If $H$ is a cohomological functor and $G$ a generator of $\sfT$ such that the $R$-module $H^{*}(G)$ is noetherian and the $R^{0}$-module $H^{i}(G)$ has finite length for each $i$, then one has an inequality:
\[
\dim \sfT\geq \dim H^{*}(Y)-1\quad\text{for each $Y\in\sfT$}.
\]
\end{theorem}

\begin{proof}[Sketch of a proof]
Under the hypotheses of the theorem, the functor $H$ is representable; this can be proved by an argument similar to that for \cite[Theorem~1.3]{BV} due to Bondal and Van den Bergh. The result is thus contained in Theorem~\ref{thm:main}.
\end{proof}

The following result is a variation on Theorems~\ref{thm:main} and \ref{thm:rep} which might be useful in some contexts.  The hypothesis on $\sfT$ holds, for example, when it is algebraic, in the sense of Keller~\cite{Kel}.

\begin{theorem}
\label{thm:algebraic}
Let $\sfT$ be a triangulated category with functorial mapping cones. If $H$ is a cohomological functor and $G$ a generator for $\sfT$ such that $H^{*}(G)$ is in $\flnoeth(R)$, for some ring $R$ acting centrally on $\sfT$, then one has an inequality:
\[
\dim \sfT\geq \dim H^{*}(Y)-1\quad\text{for each $Y\in\sfT$}.
\]
\end{theorem}

\begin{proof}[Sketch of a proof]
Since $\sfT$ has functorial mapping cones, for each $r\in R$, the construction of the Koszul object $\kos Yr$ can be made functorial. Thus the assignment $Y\mapsto \kos Yr$ defines an exact functor on $\sfT$, and therefore the assignment $Y\mapsto H(\Si^{-1}\kos Yr)$ yields a cohomological functor; let us denote it $\kos Hr$, with a caveat that it is a desuspension of what is introduced in \eqref{eq:koszul}. This functor comes equipped with a natural transformation $\theta_{r}\col \kos Hr\to H$.

Let $G$ be a generator for $\sfT$, set $c=\dim H^{*}(G)-1$ and $s=2^{c}$. Arguing as in the proof of Theorem~\ref{thm:koszul}, one can pick a sequence of elements $r_{1},\dots,r_{c}$ such that $r_{i+1}$ is filter-regular on $H_{i}^{*}(G)$, where $H_{0}=H$ and $H_{i}=\kos {H_{i-1}}{r_{i}^{s}}$ for $i\ge 1$.  One thus has natural transformations
\[
H_{c}\xra{\theta_{r_{c}}}H_{c-1}\xra{\theta_{r_{c-1}}}\cdots \xra{\theta_{r_{1}}} H_{0}
\]
satisfying, for each $i=1,\dots,c$, the following conditions:
\begin{enumerate}[{\quad\rm(1)}]
\item $\theta_{r_{i}}^{*}(G)\col H_{i}^{*}(G)\to H_{i-1}^{*}(G)$ is eventually zero;
\item $\theta_{r_{i}}^{*}(\kos G{\bsr})\col H_{i}^{*}(\kos G{\bsr})\to H_{i-1}^{*}(\kos G{\bsr})$ is surjective;
\item $H^{*}(\kos G{\bsr})$ is not eventually zero.
\end{enumerate}
It now follows from (an analogue of) the Ghost Lemma that  $\kos G{\bsr}$ is not in $\thickn c{G}$.
This implies the desired result.
\end{proof}

\section{Applications}
\label{Applications}

Let $A$ be a noetherian ring. In what follows, $\sfD^{b}(A)$ denotes the bounded derived category of finitely generated $A$-modules, with the usual structure of a triangulated category.  Following Buchweitz~\cite{Bu}, the \emph{stable derived category} of $A$ is the category
\[
\std(A) = \sfD^{b}(A)/\sfD^\per(A)\,,
\]
where $\sfD^\per(A)=\Thick(A)$ denotes the category of perfect complexes.  Here the quotient is taken in the sense of Verdier; see~\cite{V}. It has a structure of a triangulated category, for which the canonical functor $\sfD^{b}(A)\to \std(A)$ is exact.

\begin{remark}
\label{rem:central}
The quotient functor $\sfD^{b}(A)\to \std(A)$ induces a homomorphism of graded rings $Z^*(\sfD^{b}(A))\to Z^*(\std(A))$. Thus a central action of a graded commutative ring $R$ on $\sfD^b(A)$ induces a central action on $\std(A)$. In particular, for any pair of complexes $X,Y\in\sfD^b(A)$ the natural map
\[
\Ext^*_A(X,Y)=\Hom_{\sfD^{b}(A)}^{*}(X,Y)\to \Hom_{\std}^{*}(X,Y)
\]
is one of $R$-modules.
\end{remark}

\subsection*{Gorenstein rings}
A noetherian ring $A$ is called \emph{Gorenstein} if $A$ is of finite injective dimension both as a left module and a right module over itself. In the commutative case, this is more restrictive than the usual definition of a Gorenstein ring; however both definitions coincide if $A$ has finite Krull dimension.

The following result is \cite[Corollary~6.3.4]{Bu}.

\begin{lemma}
\label{lem:evbij}
Let $A$ be a noetherian Gorenstein ring. Then for each pair of complexes $X,Y\in\sfD^b(A)$ the  natural map
\[
\Hom_{\sfD^b(A)}^{n}(X,Y)\to \Hom_{\std(A)}^{n}(X,Y)
\]
induced by the quotient functor $\sfD^{b}(A)\to \std(A)$ is bijective for $n\gg 0$.\qed
\end{lemma}

The notion of complexity of a sequence of non-negative integers was recalled in the paragraph preceding Proposition~\ref{prop:flevn}.  We define the complexity of a pair $X,Y$ of complexes of $A$-modules to be the number
\[
\cx_{A}(X,Y) = \cx(\length_{Z(A)}(\Ext^{n}_{A}(X,Y)))
\]
where $Z(A)$ denotes the center of $A$.

\begin{example}
Let $A$ be an Artin $k$-algebra, and let $\fr$ denote its radical. Then every finitely generated $A$-module $M$ admits a minimal projective resolution
\[
\cdots \to P_2\to P_1\to P_0\to M\to 0
\]
and one defines the complexity of $M$ as
\[
\cx_{A}(M) = \cx(\length_{k}(P_n))\,.
\]
It is well-known that $\cx_A(M)=\cx_A(M,A/\fr)$; see \cite[A.13]{Av} or \cite[\S5.3]{BensonII} for details.
\end{example}

Recall that a ring $A$ is said to be a \emph{noetherian algebra} if there exists a commutative noetherian ring $k$ such $A$ is a $k$-algebra and a finitely generated $k$-module.

\begin{theorem}
\label{thm:goren}
Let $A$ be a noetherian algebra which is Gorenstein.  Let $X\in\sfD^{b}(A)$ be such that $\Ext_A^*(X,Y)$ is in $\flnoeth(R)$, for some graded-commutative ring $R$ acting centrally on $\sfD^b(A)$, and for all $Y\in\sfD^b(A)$.  One then has inequalities
\[
\dim \sfD^{b}(A) \geq \dim \std(A) \geq\cx_{A}(X,X)-1\,.
\]
\end{theorem}

\begin{proof}
The inequality on the left holds because $\std(A)$ is a quotient of $\sfD^{b}(A)$. The $R$-action on $\sfD^b(A)$ induces an action on $\std(A)$ by Remark~\ref{rem:central}, and the finiteness condition on $X$ as an
object of $\sfD^b(A)$ passes to the stable category $\std(A)$ because of Lemma~\ref{lem:evbij}. In particular, Lemma~\ref{lem:evbij} implies the equality
\[
\dim_R\End^*_{\sfD^b(A)}(X)=\dim_R\End^*_{\std(A)}(X)\,.
\]
Theorem~\ref{ithm:intro} now yields the inequality below
\[
\dim\std(A) \geq \dim_R \End^{*}_{\sfD^b(A)}(X)-1 = \cx_{A}(X,X)-1\,.
\]
The equality follows from Propositions~\ref{prop:flevn} and \ref{prop:dimbasechange}, where we use that $\Ext_A^n(X,X)$ is finitely generated over $Z(A)$.
\end{proof}

\subsection*{Artin algebras}
An \emph{Artin algebra} is a noetherian $k$-algebra where the ring $k$ is artinian; equivalently the center $Z(A)$ of $A$ is an artinian commutative ring and $A$ is finitely generated as a module over it. Over such rings the various finiteness conditions considered in this article coincide.

\begin{lemma}
Let $A$ be an Artin algebra and $X,Y$ objects in $\sfD^{b}(A)$. If the graded module $\Ext^{*}_{A}(X,Y)$ is in $\flnoeth(R)$ for some ring $R$ acting centrally on $\sfD^{b}(A)$, then it is noetherian and degreewise of finite length over the ring $R\otimes_{\bbz}Z(A)$.
\end{lemma}

\begin{proof}
It is easy to check that the $Z(A)$-module $\Ext^{i}_{A}(X,Y)$ has finite length for each $i$. The desired result is a consequence of this observation.
\end{proof}

For Artin algebras we are able to establish a stronger version of Theorem~\ref{thm:goren}, where one does not have to assume beforehand that the ring is Gorenstein. This is based on the following observation.

\begin{proposition}
\label{prop:fingengor}
Let $A$ be an Artin algebra with radical $\fr$. If $\Ext^*_A(A/\fr,A/\fr)$ is noetherian over some graded-commutative ring acting centrally on $\sfD^b(A)$, then $A$ is Gorenstein.
\end{proposition}

\begin{proof}
Observe that $G=A/\fr$ is a generator for $\sfD^b(A)$. Thus, an $A$-module $X$ has finite injective dimension if and only if $\Ext_A^*(G,X)$ is eventually zero, and $X$ has finite projective dimension if and only if $\Ext_A^*(X,G)$ is eventually zero. Moreover, when $\Ext_A^*(G,G)$ is eventually noetherian over some ring $R$ acting centrally on $\sfT$, then so are $\Ext_{A}^{*}(X,G)$ and $\Ext_{A}^{*}(G,X)$. In view of Lemma~\ref{lem:reven}, one may assume that $R=R^{\even}$, so Lemma~\ref{lem:dimsym} yields an equality
\[
\msup \Ext^{*}_A(X,G) = \msup \Ext^{*}_{A}(G,X)\,.
\]
Applying Proposition~\ref{prop:evn}, it follows that $X$ has finite projective dimension if and only if it   has finite injective dimension. The duality between right and left modules then implies that $A$ is Gorenstein.
\end{proof}

Recall that the Loewy length of an Artin algebra $A$, with radical $\fr$, is the least non-negative integer $n$ such that $\fr^n=0$; we denote it $\loewy(A)$.

\begin{corollary}
\label{cor:fd}
Let $A$ be an Artin algebra with radical $\fr$.  If $\Ext^{*}_A(A/\fr,A/\fr)$ is noetherian as a module over some ring acting centrally on $\sfD^{b}(A)$, then
\[
\loewy(A)\geq \dim \sfD^{b}(A) \geq \dim \std(A) \geq \cx_A(A/\fr)-1\,. 
\]
\end{corollary}

\begin{proof}
The first inequality holds because $\thickn[{}]{\loewy(A)}{A/\fr}=\sfD^{b}(A)$; see \cite[Lemma 7.35]{Rq:dim}. The rest are obtained by combining Proposition~\ref{prop:fingengor} and Theorem~\ref{thm:goren}.
\end{proof}

\begin{remark}
In view of results of Friedlander and Suslin~\cite{FS}, the preceding result applies, in particular, to the case when $A$ is a co-commutative Hopf algebra over a field $k$. In this case, the $k$-algebra $\Ext_A^*(k,k)$ acts on $\sfD^{b}(A)$ via the diagonal action.

One may specialize further to the case where $k$ is a field of characteristic $p$ and $A=kG$ is the group algebra of a finite group $G$. It follows from a theorem of Quillen~\cite{Qu} that $\cx_{kG}(k)$ equals the $p$-rank of $G$. Thus, Corollary~\ref{cor:fd} yields the following inequalities
\[
\loewy(kG)\geq \dim \std(kG) \geq \rank_{p}(G)-1\,.
\]
These estimates were first obtained in \cite{Op} using different methods.
\end{remark}

\begin{remark}
We should like to note that when $A$ is an Artin $k$-algebra which is also projective as a $k$-module, one has a natural first choice for the ring acting centrally on $\sfD^{b}(A)$, namely, the Hochschild cohomology $\HH^*(A)$ of $A$ over $k$.

Suppose that $A$ is finite dimensional over an algebraically closed field $k$.  In \cite[\S2]{EHTSS}, Erdmann et al.\ introduced the following finiteness condition: There is a noetherian graded subalgebra $H$ of $\HH^*(A)$, with $H^0=\HH^0(A)$, such that $\Ext^*_A(A/\fr,A/\fr)$ is finitely generated over $H$. This condition
has been investigated by various authors, in particular, in connection with the theory of support varieties.

The present work and \cite{AI} suggest that $\flnoeth(\bfD^b(A))=\bfD^b(A)$ is the appropriate finiteness condition on $A$; see Remark~\ref{rem:flnoeth}. In particular, the ring $R$ acting centrally on $\bfD^b(A)$ is not essential, and the emphasis shifts rather to properties of $\Ext^*_A(A/\fr,A/\fr)$ alone. While this point of view is more general, is seems also to be technically simpler and more flexible.
\end{remark}

\subsection*{Complete intersections}
For a commutative local ring $A$, with maximal ideal $\fm$ and residue field $k=A/\fm$, the number $\edim A -\dim A$ is called the \emph{codimension} of $A$, and denoted $\codim A$; here $\edim A$ is the \emph{embedding dimension} of $A$, that is to say, the $k$-vector space dimension of $\fm/\fm^{2}$.

The result below holds also without the hypothesis that $A$ is complete; see \cite{AI08}.  For the definition of a complete intersection ring, see \cite[\S2.3]{BH}.

\begin{corollary}
Let $A$ be a commutative local ring, complete with respect to the topology induced by its maximal ideal.
If $A$ is complete intersection, then
\[
\dim \sfD^{b}(A) \geq \dim \std(A) \geq \codim A - 1\,.
\]
\end{corollary}

\begin{proof} Set $c=\codim A$. The hypotheses on $A$ imply that there is a polynomial ring $A[\chi_{1},\dots,\chi_{c}]$, where the $\chi_{i}$ are indeterminates of degree $2$, acting centrally on $\sfD^{b}(A)$ with
the property that, for any pair of complexes $X,Y$ in $\sfD^{b}(A)$, the graded $A$-module $\Ext^{*}_{A}(X,Y)$ is finitely generated over $A[\chi_{1},\dots,\chi_{c}]$; see, for instance, \cite[\S7.1]{AI}. Since $A$ is complete intersection, it is Gorenstein; see \cite[Proposition~3.1.20]{BH}. Hence Theorem~\ref{thm:goren} applies and, for the residue field $k$ of $A$, yields inequalities
\[
\dim \sfD^{b}(A) \geq \dim \std(A) \geq \cx_{A}(k,k) - 1\,.
\]
It remains to note that $\cx_{A}(k,k)=\codim A$, by a result of Tate \cite[Theorem~6]{Ta}.
\end{proof}

We now apply the preceding results to obtain bounds on the representation dimension of an Artin algebra.

\subsection*{Representation dimension}
Let $A$ be an Artin algebra. The \emph{representation dimension} of $A$ is defined as 
\[
\repdim A = 
\inf \left\{\gldim\End_A(M)\left|
\begin{gathered}
\text{$M$ is a generator and a }\\
\text{cogenerator for $\mod A$}
\end{gathered}
\right. \right\}.
\]
Auslander has proved that $A$ is semi-simple if and only if $\repdim A=0$, and that $\repdim A\leq 2$ if and only if $A$ has finite representation type; see \cite{Au}. The connection between this invariant and dimensions for triangulated categories is that, when $A$ is not semi-simple, one has an inequality:
\[
\repdim A\geq \dim \std(A) +2 \,.
\]
This result is contained in \cite[Proposition~3.7]{Rq:ext}. With Theorem~\ref{thm:main}, it yields a lower bound for the representation dimension of Artin algebras:
\begin{theorem}
\label{thm:repdim}
\pushQED{\qed} 
Let $A$ be an Artin algebra that is not semi-simple, and let $\fr$ be the radical of $A$. If $\Ext^*_A(A/\fr,A/\fr)$ is noetherian as a module over some graded-commutative ring acting centrally on $\sfD^b(A)$, then
\[
\repdim A \geq\cx_{A}(A/\fr)+1\,. \qedhere
\]
\end{theorem}

With a further hypothesis that $A$ is self-injective this result was proved in \cite[Theorem~3.2]{Pb}. Arguing as \cite[Corollary~3.5]{Pb}, which again required that $A$ be self-injective, one obtains the following result relating the representation dimension of an algebra to the Krull dimension of its Hochschild cohomology ring. The hypothesis on $A/\fr\otimes_{k}A/\fr$ holds, for example, if $k$ is algebraically closed; see \cite[XVII, 6.4]{La}.

\begin{corollary}
\label{cor:Hoch}
Let  $k$ be a field, and $A$ a finite dimensional, non semi-simple, $k$-algebra with radical $\fr$, 
with $A/\fr\otimes_{k}A/\fr$ semi-simple. If $\Ext^{*}_A(A/\fr,A/\fr)$ is noetherian as a module over the Hochschild cohomology algebra $\HH^{*}(A)$ of $A$ over $k$, then
\[
\repdim A\geq \dim \HH^{*}(A)+1\,.
\]
\end{corollary}

\begin{proof}
Set $R=\HH^{*}(A)$. Given Theorem~\ref{thm:repdim} and Proposition~\ref{prop:flevn}(2), one has only to prove that
\[
\spec R = \supp_{R}\Ext^{*}_{A}(A/\fr,A/\fr) \,.
\]
This holds because the semi-simplicity of $A/\fr\otimes_{k}A/\fr$ implies that the kernel of the natural map $R\to \Ext^{*}_{A}(A/\fr,A/\fr)$ is nilpotent; see \cite[Proposition~4.4]{SS}.
\end{proof}

The inequality in the preceding result need not hold if $A/\fr\otimes_kA/\fr$ is not semi-simple, for then the kernel of the homomorphism from Hochschild cohomology to the graded center of the derived category need not be nilpotent.  This is illustrated by the following example.

\begin{example}
Let $k$ be a field of characteristic $p>0$. Assume that $k$ is not perfect, so that there is an element $a\in k$ that has no $p$th root in $k$. Let $A=k[a^{1/p}]$, the extension field obtained by adjoining the $p$th root of $a$. Since $A$ is a field, one has $\repdim A=0$ and $\dim\sfD^{b}(A)=0$.

On the other hand, it is easy to check that the Hochschild cohomology algebra of $A$ over $k$ is the polynomial algebra $A[x]$, where $|x|=1$, when $p=2$, and the graded-commutative polynomial algebra $A[x,y]$, where $|x|=1$ and $|y|=2$, when $p$ is odd.  Thus, in either case, $\dim \HH^{*}(A)=1$ holds.

The graded center of the derived category of $A$ is readily computed:
\[
\cent^{n}(\sfD^{b}(A)) =
\begin{cases}
k & \text{when $n=0$};\\
0 & \text{otherwise}\,.
\end{cases}
\]
Thus the kernel of the homomorphism $\HH^{*}(A)\to \cent^*(\sfD^{b}(A))$ is not nilpotent; compare this example with \cite[Proposition~4.4]{SS}.
\end{example}

\bibliographystyle{amsplain}
 
\end{document}